\documentclass[11pt]{article}

\usepackage{amsmath,amssymb,amsthm}
\usepackage[sort,compress]{cite}
\usepackage{xurl}
\usepackage{needspace}
\usepackage{xcolor}
\usepackage{hyperref}
\hypersetup{
  colorlinks=true,
  linkcolor=blue,
  citecolor=blue,
  urlcolor=blue
}

\newtheorem{theorem}{Theorem}[section]
\newtheorem{proposition}[theorem]{Proposition}
\newtheorem{problem}[theorem]{Problem}
\newtheorem{corollary}[theorem]{Corollary}
\newtheorem{algorithm}[theorem]{Algorithm}

\title{\textbf{On the polynomial values represented by binary quadratic forms}}
\author{%
Bogdan Grechuk\thanks{School of Computing and Mathematical Sciences, University of Leicester, LE1 7RH, United Kingdom; \texttt{bg83@leicester.ac.uk}}
\and
Jamal Agbanwa\thanks{Independent Researcher; \texttt{agbanwajamal03@gmail.com}}
}
\date{}

\begin{document}

\maketitle

\begin{abstract}
Many Diophantine equations can be reduced to the question of whether, for a given non-degenerate integral binary quadratic form $F$ and a univariate polynomial $P$ with integer coefficients, $P(x)$ can be represented by $F$ for infinitely many values of $x$. We develop a method for answering this question for certain cubic and quartic polynomials $P$, as well as for certain polynomials of the form $P(x)=R(Q(x))$, where $R(t)$ and $Q(x)$ are polynomials of degrees $3$ and $2$, respectively. Applying this method with $F(y,z)=y^2+z^2$, $R(t)=t^3-4$ and $Q(x)=x^2$, we conclude that $x^6-4$ is a sum of two squares infinitely often. In turn, this implies that the equation $y^2+x^3y+z^2+1=0$ has infinitely many integer solutions. Prior to this work, it was the shortest equation for which it was unknown whether its integer solution set is finite or infinite. We conclude with a list of the new shortest equations for which the finiteness problem remains open. All main results of this paper have been formalized in Lean using Aristotle.
\end{abstract}

\medskip
\noindent\textbf{Keywords:} Diophantine equations; integer solutions; quadratic forms; sums of squares.

\noindent\textbf{Mathematics Subject Classification (2020):} 11D09, 11D45.

\medskip

\section{Introduction}\label{sec:intro}

The study of polynomial Diophantine equations, i.e., polynomial equations with integer coefficients for which we are interested only in integer solutions, is almost as old as mathematics. Informally, a Diophantine equation is considered ``solved'' if we can ``describe'' all its integer solutions. However, the word ``describe'' here is difficult to define rigorously. Instead, one may ask the following weaker but completely formal question.

\begin{problem}\label{prob:finite}
	Given a polynomial $P(x_1,\dots,x_n)$ with integer coefficients, determine whether the equation	 
	\begin{equation}\label{eq:diofgen}
		P(x_1,\dots,x_n) = 0
	\end{equation}
	has finitely or infinitely many integer solutions. In the first case, list all its integer solutions.
\end{problem}

Problem \ref{prob:finite}, in particular, determines whether the solution set to \eqref{eq:diofgen} is non-empty, and is therefore undecidable in general \cite{MR258744}. In \cite[Chapter 8]{MR4823864}, the first author initiated a project for solving Problem \ref{prob:finite} systematically for all equations \eqref{eq:diofgen} arranged by \emph{length}. The length $l(P)$ of a polynomial $P$ written in its simplified form (where no two monomials differ solely by a constant factor) is defined as
$$
l(P) := \sum_{i=1}^k \log_2(|a_i|) + \sum_{i=1}^k d_i,
$$
where $k$ is the number of distinct monomials of $P$, $d_i$ is the degree of the $i$-th monomial, and $a_i$ is its coefficient. 
For example, the equation
\begin{equation}\label{eq:y2px3ypz2p1}
	y^2+x^3y+z^2+1=0
\end{equation}
has length
$$
l = 2 + (3+1) + 2 + 0 + \sum_{i=1}^4 \log_2 1 = 8.
$$
The length $l(P)$ is a natural parameter because it approximates the number of symbols needed to write down the polynomial $P$ as a sum of monomials, provided that all coefficients are written in binary and the power symbol is not used (e.g., we write $x^3y$ as $xxxy$).
Following \cite[Section~1.1.1]{MR4823864}, we call two polynomial Diophantine equations \emph{equivalent} if they can be transformed into one another via renaming or permuting the variables, sign changes of the variables ($x_i \to -x_i$), or multiplication of the whole equation by $-1$. Then, for any $B > 0$, there are only finitely many inequivalent equations \eqref{eq:diofgen} satisfying $l \leq B$. Hence, it is possible to arrange all equations \eqref{eq:diofgen} in non-decreasing order of their length and systematically solve Problem \ref{prob:finite} in this order.

In \cite[Section~8.3.1]{MR4823864}, it was reported that Problem \ref{prob:finite} had been solved for all Diophantine equations of length $l<9$ except \eqref{eq:y2px3ypz2p1}. It was also remarked that equation \eqref{eq:y2px3ypz2p1} can be rewritten as 
\begin{equation}\label{eq:y2px3ypz2p1aux}
	(x^3+2y)^2 + (2z)^2 = x^6-4.
\end{equation}
Together with an easy parity argument, this implies that Problem \ref{prob:finite} for equation \eqref{eq:y2px3ypz2p1} reduces to the question whether $x^6-4$ can be represented as a sum of two squares for infinitely many integers $x$. 
%In fact, famous Bunyakovsky conjecture predicts that if polynomial $P$ with integer coefficients is irreducible over rationals, has positive leading term, and there is no prime $p$ dividing $P(x)$ for all $x$, then $P(x)$ is a prime for infinitely many positive integers $x$. Applying this to $P(x)=x^6-4$ results in infinitely many (necessarily odd) values of $x$ for which $x^6-4$ is a prime, and all primes equal to $1$ modulo $4$ are sums of two squares by Fermat's theorem. However, the Bunyakovsky conjecture is a very difficult open question.

Problem \ref{prob:finite} for equation \eqref{eq:y2px3ypz2p1} was listed as an open question in the monograph \cite{MR4823864}, in the collection of open problems \cite{grechuk2024systematic}, and on MathOverflow\footnote{\url{https://mathoverflow.net/questions/454040/}}, where Daniel Loughran presented an advanced proof that $x^6-4$ is a sum of two rational squares for infinitely many rational numbers $x$. He also remarked that he ``would be very happy to see an elementary answer which avoids all this geometry''. 

The sum-of-two-squares question arising from \eqref{eq:y2px3ypz2p1aux} is a special case of a more general problem: when are the values of a one-variable polynomial represented by a fixed binary quadratic form?  More precisely, if
\[
F(y,z)=Ay^2+Byz+Cz^2, \qquad A,B,C \in {\mathbb Z},
\]
is a binary quadratic form and $P(x)$ is a polynomial with integer coefficients, one may ask whether
\begin{equation}\label{eq:FyzPx-lit}
	F(y,z)=P(x)
\end{equation}
is solvable in $(y,z)$ for infinitely many values of $x$. The classical theory of binary quadratic forms gives a very detailed description of which fixed integers are represented by a fixed form. For $y^2+z^2$, Fermat's theorem and its standard extension characterize the represented positive integers by the parity of the exponents of primes congruent to $3\pmod 4$. More generally, the theory of genera, classes, composition, automorphs and local representation is treated in standard references such as \cite{Cox1989,Buell1989,Cassels1978,OMeara1963,Serre1973}. From an asymptotic point of view, Landau's theorem on sums of two squares was generalized by Bernays to primitive positive definite binary quadratic forms \cite{Bernays1912}; see also James \cite{James1938} for the distribution of integers represented by quadratic forms. These results describe the represented integers in the aggregate, but they do not by themselves decide whether a prescribed sparse sequence $P({\mathbb Z})$ meets the represented set infinitely often.

When $P$ has degree at most two, equation \eqref{eq:FyzPx-lit} is itself a quadratic Diophantine equation in three variables. In principle, such equations are covered by the general algorithm of Grunewald and Segal for integer solutions of quadratic equations \cite{MR2055712}. There is also a direct analytic theorem of Friedlander and Iwaniec \cite{FriedlanderIwaniec1978}: for a quadratic polynomial $P(x)=ax^2+bx+c$ with $a>0$ and negative discriminant, and a binary quadratic form satisfying natural local conditions, they obtain lower bounds of order $X(\log X)^{-1/2}$ for the number of $0<x\le X$ for which $P(x)$ is represented. Their proof uses the half-dimensional sieve \cite{Iwaniec1976}. More recently, Cao and Huang obtained geometric-sieve results for affine quadrics and, as an auxiliary application, density estimates for quadratic-polynomial values represented by positive definite binary quadratic forms \cite{CaoHuang2022}. These results give powerful analytic information in the quadratic case, whereas the main difficulty of the present paper begins with cubic and higher-degree $P$.

When $P$ has degree three, the affine surface defined by \eqref{eq:FyzPx-lit} is closely related to Châtelet surfaces and to pencils of conics. Iwaniec and Munshi studied rational points on equations of the form \eqref{eq:FyzPx-lit}, where $F$ is an irreducible binary quadratic form and $P$ is an irreducible cubic polynomial \cite{IwaniecMunshi2010}. This is one of the central analytic papers on cubic polynomials represented by binary quadratic forms, but their result is primarily about rational points rather than integral points. A recent result more directly concerned with integral representations is due to Iyer \cite{Iyer2025}, who proves lower bounds for the frequency with which certain irreducible monic cubic polynomials are sums of two squares; for example, his methods apply to polynomials such as $P(x)=x^3+h$ under explicit hypotheses on $h$. His proof uses units in degree-six number fields. In addition, \cite[Section 5.4]{MR4823864} develops powerful methods for solving Problem \ref{prob:finite} for equations of the form \eqref{eq:FyzPx-lit} with cubic $P$ and certain multiplicative forms $F$.

There is also an integral-point literature on affine Châtelet surfaces. Gundlach considered integral Brauer--Manin obstructions for equations $y^2+z^2+x^k=m$, equivalently $y^2+z^2=m-x^k$, and obtained conditional completeness results under Schinzel's hypothesis \cite{Gundlach2013}. Berg studied Brauer--Manin obstructions to integral points on affine surfaces $y^2-az^2=P(x)$ \cite{Berg2017}. These results are close in spirit to the study of Problem \ref{prob:finite} for equation \eqref{eq:FyzPx-lit}; however, they are primarily local-global statements, and some are conditional on Schinzel's hypothesis.

In summary, previous results either solve the degree-$\le 2$ case by the theory of quadratic equations, prove rational-point theorems or partial integral results in the cubic case, or establish conditional local-global results. They do not directly provide an elementary construction proving that the degree-six representation problem $y^2+z^2=x^6-4$ has infinitely many integral solutions.

Given that the project of solving Problem \ref{prob:finite} for all Diophantine equations ordered by length had stalled at equation \eqref{eq:y2px3ypz2p1}, recent work has also considered less demanding questions for special families of equations. For example, if one asks only whether the solution set is non-empty and restricts attention to cubic equations, then this problem has recently been solved for all cubic equations of length $l < 12+\log_2 3 \approx 13.6$ \cite{grechuk2026shortest}. Instead of restricting the degree, one may restrict the number of variables \cite[Section 3]{MR4823864}, the number of monomials \cite{GRECHUK202369}, or consider only equations of Fermat type \cite{ratcliffe2025generalized}. See \cite[Section 8]{grechuk2024systematic} for a dynamically updated list of recently solved equations in various categories.

In this work, we develop a general elementary method for proving, in many cases, that values of a polynomial $P(x)$ are represented by a given non-degenerate integral binary quadratic form $F$ for infinitely many values of $x$. The method can be tried if either (i) $P$ is cubic, (ii) $P$ is quartic, or (iii) $P$ is a polynomial of degree $6$ representable as
\begin{equation}\label{eq:qx3pc}
	P(x)=R(Q(x)),
\end{equation}
where $R(t)$ and $Q(x)$ are cubic and quadratic polynomials with integer coefficients, respectively. The method is not guaranteed to work for all polynomials of the forms (i)--(iii), but it works well in many interesting examples. In particular, applying it to $F(y,z)=y^2+z^2$, $R(t)=t^3-4$ and $Q(x)=x^2$, we prove that $x^6-4$ is a sum of two squares infinitely often, and conclude that equation \eqref{eq:y2px3ypz2p1} has infinitely many integer solutions.

The resolution of equation \eqref{eq:y2px3ypz2p1} completes the project of solving Problem \ref{prob:finite} for all equations of length $l<9$, and the next step is to study equations of length $l=9$. Before this work, there were $24$ open equations of this length; see \cite[Version 7]{grechuk2024systematic}. Four of them are similar to equation \eqref{eq:y2px3ypz2p1}, and we resolve these equations by the same method.

The remainder of the paper is organised as follows. Section \ref{sec:method} describes a method for proving that values of certain polynomials of the form \eqref{eq:qx3pc} are sums of two squares infinitely often. Section \ref{sec:applications} applies this method to solve Problem \ref{prob:finite} for equation \eqref{eq:y2px3ypz2p1}, as well as for four equations of length $l=9$ that were left open in \cite{MR4823864,grechuk2024systematic}. Section \ref{sec:otherforms} records an extension of the same tangent construction to general binary quadratic forms.

\textbf{Declaration on the use of AI.}
The main results of this paper emerged from extended, multi-round interactions between the authors and ChatGPT 5.5 Pro. As is often the case in collaborative work, it is difficult to separate the individual contributions precisely. For example, in one of the rounds, ChatGPT produced the first essentially correct proof that equation \eqref{eq:y2px3ypz2p1} has infinitely many integer solutions. That proof relied on substantial hints from the authors; at the same time, those hints were themselves influenced by the model's responses in earlier rounds, and the development was therefore iterative. The proofs initially produced by ChatGPT were almost always unnecessarily long, and the authors substantially revised, shortened, and streamlined them. The final mathematical arguments and exposition were written, checked, and approved by the authors, who take full responsibility for the content of the paper. ChatGPT was also used for minor language polishing.

In addition, we used Aristotle~\cite{achim2025aristotle} to formalize all of the main results of this paper in the Lean theorem prover~\cite{demoura2021lean4}. The resulting Lean project and proof file are available on GitHub\footnote{Project snapshot: \url{https://github.com/JAgbanwa/PolynomialDiophantineEquations/tree/f3203621e90e5b1b087ca473ba8e8a7a2856b009/sumsquares_revised}. \\
Proof file: \url{https://github.com/JAgbanwa/PolynomialDiophantineEquations/blob/a89faa44bce3de44fcfcc3ad17785a0c6f14f2b2/sumsquares_revised/RequestProject/Main.lean}.}.
%The project pins Lean~4.28.0 and Mathlib commit
%\texttt{8f9d9cff6bd728b17a24e163c9402775d9e6a365}; it includes
%\texttt{formalization.yaml}, the statement interface \texttt{Challenge.lean},
%\texttt{Solution.lean}, and a comparator configuration for sixteen advertised results.
The proof sources contain no proof placeholders or added axioms and use only the standard Lean axioms.
%\texttt{propext}, \texttt{Classical.choice}, and \texttt{Quot.sound}.
%In particular, the four non-squareness checks in Section~\ref{sec:applications} use integer bounds and \texttt{norm\_num}, not \texttt{native\_decide}.
%The deliberate placeholders in \texttt{Challenge.lean} specify statements for comparison;
%they are not used by the proofs in \texttt{Solution.lean}.
%The formalization verifies the correctness implications of Algorithms~\ref{alg:main}
%and~\ref{alg:generalform}, conditional on the supplied seed data; it does not implement
%the search or the quadratic-equation decision procedures discussed below.

\section{The sum of two squares case}\label{sec:method}

In this section we study Problem \ref{prob:finite} for equations of the form
$$
	y^2+z^2=P(x),
$$
where $P(x)$ is given by \eqref{eq:qx3pc}. We present the sum of two squares case first, because it is of independent interest, suffices for applications in Section \ref{sec:applications}, and, most importantly, because the resulting algorithm is easier to understand and apply. 

Let ${\cal S}_2$ be the set of positive integers representable as sums of two integer squares. Equivalently, \cite[Theorem 5.11]{MR4823864}, ${\cal S}_2$ is the set of positive integers in whose prime factorizations every prime $p$ with $p\equiv 3 \pmod{4}$ occurs with an even exponent. This characterization immediately implies that

\begin{itemize}
	\item[(*)] if $a,b$ are positive integers and ${\cal S}_2$ contains two of the integers $a,b,ab$, then ${\cal S}_2$ contains all three of them. 
\end{itemize}

In this section, we study the following special case of Problem \ref{prob:finite}.

\begin{problem}\label{prob:sumsquares}
	Given a polynomial $P(x)$ with integer coefficients, determine whether there are infinitely many integers $x$ such that $P(x) \in {\cal S}_2$.
\end{problem}

If the degree of $P(x)$ is at most two, the problem reduces to deciding whether the quadratic Diophantine equation $y^2+z^2=P(x)$ has infinitely many integer solutions; there is a general algorithm for answering this question for all quadratic equations \cite{MR2055712}. A powerful method for solving this problem for cubic polynomials $P(x)$ is described in \cite[Section 5.2.3]{MR4823864}. Here, we present a method that works for many cubic and quartic polynomials, as well as polynomials of the form \eqref{eq:qx3pc} with cubic $R(t)$ and quadratic $Q(x)$. To unify these cases, we assume that $P(x)$ has the form \eqref{eq:qx3pc} with $(\deg R, \deg Q)$ equal to either $(3,1)$, or $(4,1)$, or $(3,2)$.

To prove that $P(x)$ in \eqref{eq:qx3pc} is a sum of two squares infinitely often, we need to find infinitely many integers $t$ such that 
\begin{itemize}
	\item[(i)] $R(t) \in {\cal S}_2$, and
	\item[(ii)] $t=Q(x)$ for some integer $x$. 
\end{itemize}

To find many integers $t$ satisfying (i), we use the following idea. For any fixed integer $u \in {\mathbb Z}$, there is a unique polynomial $D_u(t)$ with integer coefficients satisfying
\begin{equation}\label{eq:taylorDu}
	R(t)=R(u)+R'(u)(t-u)+(t-u)^2 D_u(t).
\end{equation}
Multiplying both sides of \eqref{eq:taylorDu} by $4R(u)$ and rearranging, we obtain
\begin{equation}\label{eq:generalTangentIdentity}
	4R(u)R(t)=(2R(u)+R'(u)(t-u))^2+(t-u)^2 (4R(u)D_u(t)-R'(u)^2).
\end{equation}

This identity leads to the following sufficient condition for guaranteeing that $R(Q(x))\in\mathcal S_2$ infinitely often.

\begin{proposition}\label{prop:sumsquarecriterion}
Let $R(t)$ and $Q(x)$ be non-constant polynomials with integer coefficients, and let
$u\in\mathbb Z$ satisfy $R(u)\in\mathcal S_2$. Define $D_u$ by
\eqref{eq:taylorDu}. If the auxiliary equation
\begin{equation}\label{eq:auxquad2}
    4R(u)D_u(Q(x))-R'(u)^2=v^2
\end{equation}
has infinitely many integer solutions $(x,v)$, then $R(Q(x))\in\mathcal S_2$
for infinitely many integers $x$.
\end{proposition}
\begin{proof}
For each fixed $x$, equation~\eqref{eq:auxquad2} is satisfied for at most two integer values of $v$, so its infinitely many solutions have infinitely many distinct $x$-coordinates. Equation $R(Q(x))=0$ has only finitely many roots, hence there are infinitely many integer solutions $(x,v)$ to \eqref{eq:auxquad2} such that $R(Q(x))\neq 0$. For any such pair $(x,v)$, substituting $t=Q(x)$ into \eqref{eq:generalTangentIdentity} and using \eqref{eq:auxquad2} shows that 
$$
	4R(u)R(Q(x))=(2R(u)+R'(u)(Q(x)-u))^2+((Q(x)-u)v)^2 
$$
is a sum of two integer squares. Since $R(u)>0$, and $R(Q(x)) \neq 0$, this implies that $4R(u)R(Q(x)) \in \mathcal S_2$. Also, $4 \in \mathcal S_2$ and $R(u) \in \mathcal S_2$ imply that $4R(u)\in\mathcal S_2$ by (*), so another application of (*) gives
$R(Q(x))\in\mathcal S_2$.
\end{proof}

We have $\deg D_u(t) = \deg R(t) - 2$. Hence, the left-hand side of \eqref{eq:auxquad2} is linear in $x$ when $(\deg R, \deg Q)=(3,1)$. It is quadratic in $x$ when $(\deg R, \deg Q)$ is equal to either $(4,1)$ or $(3,2)$. Thus, in all cases, equation \eqref{eq:auxquad2} takes the form 
\begin{equation}\label{eq:quadxt}
	ax^2+bx+c = v^2
\end{equation} 
for some integer coefficients $a,b,c$.

\begin{proposition}\label{prop:sumaux}
Let $a,b,c\in\mathbb Z$. Assume that
\begin{itemize}
	\item[(a)] either $a=0$, or $a>0$ and $a$ is not a perfect square, 
	\item[(b)] $b^2 - 4 a c \neq 0$, and
	\item[(c)] equation \eqref{eq:quadxt} has an integer solution $(x_0,v_0)$.
\end{itemize}
Then equation \eqref{eq:quadxt} has infinitely many integer solutions $(x,v)$,
with infinitely many distinct values of $x$.
\end{proposition}
\begin{proof}
If $a>0$ is not a perfect square, this is \cite[Proposition 5.4]{MR4823864}. If $a=0$, then condition (b) gives $b\neq 0$; starting from one solution $(x_0,v_0)$, any integer $v\equiv v_0 \pmod{|b|}$ gives an integer $x=(v^2-c)/b$, so there are infinitely many solutions. In either case, the set of corresponding $x$-coordinates is infinite, since \eqref{eq:quadxt} has at most two possible values of $v$ for each fixed $x$.
\end{proof}

%\noindent\emph{Formalization convention.}
%The Lean predicate for a sum of two squares also includes zero. Its version of
%Proposition~\ref{prop:sumsquarecriterion} consequently does not need the
%non-constancy assumption. For the polynomial cases considered here, deleting the
%finitely many zeros, as in the proof above, gives exactly the conclusion with our
%positive-integer set $\mathcal S_2$.

%\paragraph{Checking the seed condition.}\label{par:seed2}
For a fixed equation \eqref{eq:quadxt} passing (a) and (b), condition (c) is a question of integral solubility for a quadratic equation, so it is decidable by the algorithm of Grunewald and Segal~\cite{MR2055712}. However, for binary quadratic equations, much simpler and faster methods are available, see \cite[Section 3.1]{MR4823864}. In fact, condition (c) has a certificate verifiable in polynomial time and therefore can always be checked in deterministic exponential time \cite{Lagarias1979}. In the linear case $a=0$, there is a particularly simple finite test: check $v^2\equiv c\pmod{|b|}$ for $0\leq v<|b|$, where $b\ne0$ by (b); a successful residue gives $x=(v^2-c)/b$.

This gives the following recipe for solving Problem \ref{prob:sumsquares} for certain polynomials $P(x)$ of the form \eqref{eq:qx3pc}.

\begin{algorithm}\label{alg:main}~ 
	\begin{itemize}
		\item \textbf{Input:} Polynomials $R(t)$ and $Q(x)$ with integer coefficients such that $(\deg R, \deg Q)$ is equal to either $(3,1)$, or $(4,1)$, or $(3,2)$.
		\item \textbf{Output:} If the algorithm terminates, then $P(x)=R(Q(x)) \in {\cal S}_2$ for infinitely many integers $x$. 
	\end{itemize}
	\begin{enumerate}
		\item By direct search, find $u$ such that $R(u) \in {\cal S}_2$. 
		\item For this $u$, write down equation \eqref{eq:auxquad2} in the form \eqref{eq:quadxt} and check conditions (a), (b) and (c) of Proposition~\ref{prop:sumaux}.
		\item If all conditions (a)--(c) are satisfied, STOP. Otherwise return to Step 1 and try to find a different $u$. 
	\end{enumerate}
\end{algorithm}

This algorithm does not always terminate, because a suitable $u$ is not guaranteed to exist. However, it works in many interesting examples, as we illustrate in the next section.

\section{Application to the shortest open equations}\label{sec:applications}

In this section, we apply Algorithm \ref{alg:main} to solve Problem \ref{prob:sumsquares} for $P(x)=x^6+f$ for $f\in\{8,5,-3,-4\}$. This resolves Problem \ref{prob:finite} for equation \eqref{eq:y2px3ypz2p1} and for four further equations of length $l=9$.

In all these examples, we will apply Algorithm \ref{alg:main} with input 
\begin{equation}\label{eq:inputPQ}
Q(x)=x^2 \quad \text{and} \quad R(t)=t^3+f \quad \text{for} \,\,f\in\{8,5,-3,-4\},
\end{equation}
so that $P(x)=R(Q(x))=(x^2)^3+f=x^6+f$. In this case, $R(u)=u^3+f$ and $R'(u)=3u^2$, hence, from \eqref{eq:taylorDu},
$$
	D_u(t) = \frac{R(t)-R(u)-R'(u)(t-u)}{(t-u)^2} = t+2u,
$$
and equation \eqref{eq:auxquad2} reduces to
$$
	4(u^3+f)(x^2+2u)-9u^4 = v^2,
$$
or equivalently
\begin{equation}\label{eq:auxquad3}
	4(u^3+f)x^2-u(u^3-8f) = v^2.
\end{equation}

\Needspace{10\baselineskip}
We begin with equation \eqref{eq:y2px3ypz2p1}.

\begin{corollary}\label{prop:y2px3ypz2p1}
	Equation \eqref{eq:y2px3ypz2p1} has infinitely many integer solutions.
\end{corollary}
\begin{proof}
	Equation \eqref{eq:y2px3ypz2p1} is equivalent to \eqref{eq:y2px3ypz2p1aux}, whose right-hand side is $x^6-4=(x^2)^3-4$. We will apply Algorithm \ref{alg:main} with input \eqref{eq:inputPQ} for $f=-4$. The smallest positive $u$ satisfying $R(u)=u^3-4\in {\cal S}_2$ is $u=2$, for which $u^3-4=4=2^2+0^2$. However, for $u=2$, equation \eqref{eq:auxquad3} reduces to $16x^2-80=v^2$. Hence, $a=16$ in \eqref{eq:quadxt} is a perfect square, and condition (a) fails.
	
	A direct search shows that, for example\footnote{The smallest positive integer $u$ satisfying all the conditions is $u=50$. However, the solution $(x_0,v_0)$ for $u=50$ is much larger, and, for this reason, we have decided to use $u=162$.}, $u = 162$ satisfies all the conditions. Indeed, 
	$$
		R(u)=u^3+f=162^3-4=350^2+2032^2 \in {\cal S}_2,
	$$
	while equation \eqref{eq:auxquad3} reduces to
	\begin{equation}\label{eq:eq1aux}
		17{,}006{,}096 \, x^2 - 688{,}752{,}720 = v^2.
	\end{equation}
	This is equation \eqref{eq:quadxt} with $(a,b,c)=(17{,}006{,}096,\, 0, \, -688{,}752{,}720)$. Then $a>0$, $a$ is not a perfect square, $b^2 - 4 a c \neq 0$, and equation \eqref{eq:eq1aux} has a solution 
	$$
		x_0 = 22{,}108{,}343{,}594{,}783{,}571, \quad v_0 = 91{,}171{,}377{,}945{,}572{,}295{,}096.
	$$   
	Hence, Proposition~\ref{prop:sumaux} gives infinitely many integer solutions of \eqref{eq:eq1aux}, and Proposition~\ref{prop:sumsquarecriterion} gives $x^6-4\in\mathcal S_2$ for infinitely many integers $x$. No such $x$ can be even: if $x=2w$, then $x^6-4=4(16w^6-1)$, and membership of $x^6-4$ in ${\cal S}_2$ would imply, by property (*), that $16w^6-1\in {\cal S}_2$. This is impossible, because this number is congruent to $3$ modulo $4$. Thus all such $x$ are odd. Then we have $x^6-4=A^2+B^2$ for some integers $A,B$ of opposite parity. By swapping $A$ and $B$ if needed, we may assume that $A$ is odd and $B$ is even, in which case we can write $A=x^3+2y$ and $B=2z$ for some integers $y$ and $z$. Hence, equation \eqref{eq:y2px3ypz2p1aux}, or, equivalently, equation \eqref{eq:y2px3ypz2p1}, has infinitely many integer solutions.  
\end{proof}

Corollary \ref{prop:y2px3ypz2p1} finishes the resolution of Problem \ref{prob:finite} for all polynomial Diophantine equations of length $l < 9$. Before this work, there were exactly $24$ inequivalent equations of length $l=9$ for which Problem \ref{prob:finite} was open; see \cite[Version 7]{grechuk2024systematic}. 
They consist of the four equations
\begin{equation}\label{eq:eq91}
	y^2+x^3y+z^2-2=0,
\end{equation}
\begin{equation}\label{eq:eq92}
	y^2+x^3y+z^2+z-1=0,
\end{equation}
\begin{equation}\label{eq:eq93}
	y^2+x^3y+z^2+z+1=0,
\end{equation}
\begin{equation}\label{eq:eq94}
	y^2+x^3y+y+z^2+1=0,
\end{equation}
two further equations $z^2+y^2z+2x^3+1=0$ and $x^3+x^2y^2+z^2+1=0$ solved by the second author using different methods while this paper was under review, see \cite[Section 8.8]{grechuk2024systematic}, and the $18$ equations listed in Table \ref{tab:l9openfin}.

\begin{table}
	\begin{center}
		\begin{tabular}{ |c|c|c| } 
			\hline
			Equation & Equation & Equation \\ 
			\hline\hline
			$z^2+y^2z+x^3-2=0$  & $y(x^3-z^2)=2x-1$ & $x^3y^2=z^4+1$ \\ 
			\hline
			$z^2+y^2z+x^3-x-1=0$  & $y(x^3-z^2)=2x+1$ & $x^4y^3=z^2+1$ \\ 
			\hline
			$z^2+y^2z+x^3y+1=0$ & $x^3y^2=z^3+2$ & $y^3-y=x^4-x$ \\ 
			\hline
			$x^2y+y^2z+z^2x=1$ & $x^3y^2=z^3-z+1$ & $y^3+y=x^4+x$ \\ 
			\hline
		 	$x^4+y^3+z^2+1=0$ & $x^3y^2=z^3+z+1$ & $x^4+xy+y^3-1=0$ \\ 
			\hline
			$y(x^3-z^2)=x^2+1$ & $x^3y^2=2z^3+1$ & $x^4+xy+y^3+1=0$ \\ 
			\hline
		\end{tabular}
		\caption{\label{tab:l9openfin} The eighteen equations of length $l=9$ for which Problem \ref{prob:finite} remains open after this work.}
	\end{center} 
\end{table} 

Equations \eqref{eq:eq91}-\eqref{eq:eq94} are similar to equation \eqref{eq:y2px3ypz2p1}, and can be solved by the same method.

\begin{corollary}\label{prop:eq9194}
	Each of the equations \eqref{eq:eq91}-\eqref{eq:eq94} has infinitely many integer solutions.
\end{corollary}
\begin{proof}
	Equations \eqref{eq:eq91}-\eqref{eq:eq93} can be rewritten in the form
	\begin{equation}\label{eq:eq9194aux}
		(x^3+2y)^2 + (2z+e)^2 = x^6+f,
	\end{equation}
	where $(e,f)=(0,8)$ for \eqref{eq:eq91}, $(e,f)=(1,5)$ for \eqref{eq:eq92}, and $(e,f)=(1,-3)$ for \eqref{eq:eq93}. Hence, we need to prove that $x^6+f \in {\cal S}_2$ infinitely often for $f=8$, $f=5$, and $f=-3$. For these values of $f$, a direct search shows that Algorithm \ref{alg:main} with input \eqref{eq:inputPQ} works for $u=8$, $u=2$, and $u=2$, respectively. Indeed, 
	$$
		8^3+8 = 520 = 22^2+6^2, \quad 2^3+5 = 13 = 3^2+2^2, \quad 2^3-3 = 5 = 2^2+1^2
	$$
	are all sums of two squares. Further, for $(f,u)=(8,8)$, $(5,2)$, and $(-3,2)$, equation \eqref{eq:auxquad3} takes the forms
	$$
		2080 x^2-3584 = v^2, \quad 52 x^2 + 64 = v^2, \quad \text{and} \quad 20 x^2 - 64 = v^2,
	$$
	respectively. We claim that all these equations have infinitely many integer solutions $(x,v)$ with $x=2w$ even, or, equivalently, that equations  
	$$
		8320 w^2-3584 = v^2, \quad 208 w^2 + 64 = v^2, \quad \text{and} \quad 80 w^2 - 64 = v^2,
	$$
	have infinitely many solutions in integers $(w,v)$. Indeed, these are equations of the form \eqref{eq:quadxt} with $(a,b,c)=(8320,0,-3584)$, $(208,0,64)$, and $(80,0,-64)$, respectively. In all cases, $a>0$, $a$ is not a perfect square, and $b^2 - 4 a c \neq 0$. Further, the listed equations have integer solutions $(w_0, v_0) = (6, 544)$, $(w_0, v_0) = (3, 44)$, and $(w_0, v_0) = (1, 4)$, respectively. Hence, the claim follows from Proposition~\ref{prop:sumaux}. Then applying Proposition~\ref{prop:sumsquarecriterion} with $R(t)=t^3+f$ and $Q(w)=4w^2$ shows that for each $f\in\{8,5,-3\}$ there are infinitely many even $x=2w$ for which $x^6+f\in\mathcal S_2$. 
	
	For $f=8$, if $x^6+8=A^2+B^2$ for even $x$, then reduction modulo $4$ shows that both $A$ and $B$ are even; hence we may write $A=x^3+2y$ and $B=2z$ for some integers $y,z$. This proves that equation \eqref{eq:eq9194aux} with $(e,f)=(0,8)$ has infinitely many integer solutions. Similarly, for $f \in \{5,-3\}$, if $x^6+f=A^2+B^2$ for even $x$, then $A$ and $B$ have opposite parity, and we can assume that $A$ is even and $B$ is odd. Then we can write $A=x^3+2y$ and $B=2z+1$ for some integers $y,z$, which proves that equation \eqref{eq:eq9194aux} with $(e,f)=(1,5)$ and $(e,f)=(1,-3)$ has infinitely many integer solutions. This finishes the proof of the corollary for equations \eqref{eq:eq91}-\eqref{eq:eq93}. 
	
	Equation \eqref{eq:eq94} can be rewritten as 
	\begin{equation}\label{eq:eq94aux}
		(x^3 + 2y + 1)^2 + (2z)^2 = (x^3 + 1)^2 - 4,
	\end{equation}
	hence we need to prove that $P(x) = (x^3 + 1)^2 - 4 \in {\cal S}_2$ infinitely often. Put $x=-w^2$. Then
	$$
		P(-w^2) = (-w^6 + 1)^2 - 4 = (w^6-3)(w^6+1).
	$$
	The case $f=-3$ above shows that $w^6-3 \in {\cal S}_2$ for infinitely many even integers $w$. Also, $w^6+1=(w^3)^2+1^2 \in {\cal S}_2$ for all $w$. Hence, applying property (*) with $a=w^6-3$ and $b=w^6+1$, we conclude that $P(-w^2) \in {\cal S}_2$ for infinitely many even $w$. The map $w\mapsto-w^2$ has at most two preimages for each value, so these infinitely many $w$ give infinitely many distinct $x$. 
	If $w$ is even, then $x=-w^2$ is even as well, and equation $P(x) = (x^3 + 1)^2 - 4 = A^2 + B^2$ implies that $A$ and $B$ have opposite parity. By swapping $A$ and $B$ if needed, we can assume that $A$ is odd and $B$ is even. Then we can write $A=x^3+2y+1$ and $B=2z$ for some integers $y,z$, which proves that equation \eqref{eq:eq94aux}, or, equivalently, equation \eqref{eq:eq94}, has infinitely many integer solutions.
\end{proof}

After equations \eqref{eq:eq91}-\eqref{eq:eq94} are resolved in Corollary \ref{prop:eq9194}, the eighteen equations listed in Table \ref{tab:l9openfin} are now the only equations of length $l \leq 9$ for which Problem \ref{prob:finite} remains open. The reader is invited to solve any of these equations. See also \cite{grechuk2024systematic} for the continually updated full list of the current open equations in this project in various categories. 

A natural direction for further work is to develop methods for solving equations of the form \eqref{eq:qx3pc} for which the resulting auxiliary equation \eqref{eq:auxquad2} has local obstructions. For example, we do not know whether equation
\[
y^2+z^2=x^6+3
\]
has infinitely many integer solutions. Applying Algorithm \ref{alg:main} with $R(t)=t^3+3$ and $Q(x)=x^2$, any $u$ satisfying Step~1 must be odd, since even $u$ gives $u^3+3\equiv3\pmod 4$. For odd $u$, the auxiliary equation \eqref{eq:auxquad3} is
\[
4(u^3+3)x^2-u(u^3-24)=v^2,
\]
and its left-hand side is congruent to $3\pmod 4$, which is impossible for a square. Thus, Algorithm \ref{alg:main} cannot resolve this equation, and new ideas are needed.

\section{General binary quadratic forms}\label{sec:otherforms}

In this section, we show that the method of Section \ref{sec:method} extends if the form
$y^2+z^2$ is replaced by an arbitrary non-degenerate integral binary quadratic form
\begin{equation}\label{eq:generalbinaryform}
	F(y,z)=Ay^2+Byz+Cz^2, \qquad A,B,C\in {\mathbb Z}.
\end{equation}
Recall that $F$ is called non-degenerate if its discriminant
$$
\Delta=B^2-4AC
$$
satisfies
\begin{equation}\label{eq:nonddegF}
	\Delta\ne 0.
\end{equation}
We briefly discuss the degenerate case at the end of the section.

Let $R(t)$ and $Q(x)$ be univariate polynomials with integer coefficients, and define $D_u(t)$ by \eqref{eq:taylorDu}.  Suppose that, for some integer $u$, the value $R(u)$ is non-zero and is represented by $F$, say
\begin{equation}\label{eq:basepointF}
	m:=R(u)=F(p,q)=Ap^2+Bpq+Cq^2\ne 0, \quad \text{for some} \quad p,q\in {\mathbb Z}.
\end{equation}

The core of the method is the following tangent-line idea.  Write
$t=u+s$ and try to represent $R(t)$ as $F(y,z)$ by points lying on an affine line through the known point $(p,q)$:
\begin{equation}\label{eq:yzdef}
y=p+s\lambda, \qquad z=q+s\mu .
\end{equation}
On the one hand, by \eqref{eq:taylorDu},
$$
R(t) = R(u+s)=m+sr+s^2D_u(u+s),
$$
where $r:=R'(u)$. On the other hand,
\begin{equation}\label{eq:expansion}
	F(y,z) = F(p+s\lambda,q+s\mu)=m+s\bigl(2Ap\lambda+B(p\mu+q\lambda)+2Cq\mu\bigr)+s^2F(\lambda,\mu).
\end{equation}
Thus, to ensure that $R(t)=F(y,z)$, it is sufficient to have
\begin{equation}\label{eq:conditions}
2Ap\lambda+B(p\mu+q\lambda)+2Cq\mu=r
\quad \text{and} \quad
F(\lambda,\mu)=D_u(t).
\end{equation}
The first of these equations is a line in the $(\lambda,\mu)$-plane.  A convenient parametrization of this line is
\begin{equation}\label{eq:lambdamu}
\lambda=\frac{rp+v(Bp+2Cq)}{2m},
\qquad
\mu=\frac{rq-v(2Ap+Bq)}{2m},
\end{equation}
where $v$ is a parameter; geometrically, the terms multiplied by $v$ give the tangent direction to the conic $F(y,z)=m$ at $(p,q)$.  Substituting this parametrization into the second condition in \eqref{eq:conditions} gives
$$
4mD_u(t)-r^2=-\Delta v^2.
$$
Substituting $t=Q(x)$ in this equation, and adding the congruence conditions to ensure that $\lambda$ and $\mu$ are integers, results in the following proposition.

\begin{proposition}\label{prop:generalformmethod}
	Assume \eqref{eq:nonddegF} and \eqref{eq:basepointF}, let $r=R'(u)$, and let $D_u(\cdot)$ be defined in \eqref{eq:taylorDu}.  If there are infinitely many integer pairs $(x,v)$ satisfying
	\begin{equation}\label{eq:generalformaux}
		4mD_u(Q(x))-r^2=-\Delta v^2
	\end{equation}
	and the congruences
	\begin{equation}\label{eq:generalformcong}
		\begin{split}
			rp+v(Bp+2Cq)&\equiv 0 \pmod {2|m|},\\
			rq-v(2Ap+Bq)&\equiv 0 \pmod {2|m|},
		\end{split}
	\end{equation}
	then the equation
	$$
	F(y,z)=R(Q(x))
	$$
	is solvable in integers $(y,z)$ for infinitely many integers $x$.  More precisely, for every pair $(x,v)$ satisfying \eqref{eq:generalformaux} and \eqref{eq:generalformcong}, one obtains a solution by putting
	\begin{equation}\label{eq:generalformYZ}
		\begin{split}
			y&=p+(Q(x)-u)\frac{rp+v(Bp+2Cq)}{2m},\\
			z&=q+(Q(x)-u)\frac{rq-v(2Ap+Bq)}{2m}.
		\end{split}
	\end{equation}
\end{proposition}
\begin{proof}
	Let $(x,v)$ satisfy \eqref{eq:generalformaux} and \eqref{eq:generalformcong}. Put
	$
	s=Q(x)-u,
	$
	and define $\lambda$ and $\mu$ by \eqref{eq:lambdamu}. The congruences \eqref{eq:generalformcong} imply that $\lambda,\mu\in{\mathbb Z}$. Then \eqref{eq:generalformYZ} is just \eqref{eq:yzdef}.
	Expanding $F(y,z)$ gives \eqref{eq:expansion}.
	
	Using $m=F(p,q)=Ap^2+Bpq+Cq^2$ and $\Delta=B^2-4AC$, direct substitution of the expressions \eqref{eq:lambdamu} for $\lambda$ and $\mu$ gives
	$$
	2Ap\lambda+B(p\mu+q\lambda)+2Cq\mu=r
	\quad
	\text{and}
	\quad
	4mF(\lambda,\mu)=r^2-\Delta v^2.
	$$
	By \eqref{eq:generalformaux}, the second identity is equivalent to
	$
	F(\lambda,\mu)=D_u(Q(x)).
	$
	Hence, \eqref{eq:expansion} implies that
	$$
	F(y,z)=m+sr+s^2D_u(Q(x)).
	$$
	Since $s=Q(x)-u$, the Taylor identity \eqref{eq:taylorDu} gives
	$$
	m+sr+s^2D_u(Q(x))=R(Q(x)).
	$$
	Therefore \eqref{eq:generalformYZ} gives an integer solution of
	$$
	F(y,z)=R(Q(x)).
	$$
	Finally, for any fixed $x$, equation \eqref{eq:generalformaux} has at most two possible values of $v$, because $\Delta\ne0$. Thus infinitely many pairs $(x,v)$ satisfying \eqref{eq:generalformaux} have infinitely many distinct $x$-coordinates.
\end{proof}

It remains to investigate when the auxiliary equation \eqref{eq:generalformaux} has infinitely many solutions satisfying the congruences \eqref{eq:generalformcong}. As in Section \ref{sec:method}, we assume that the degrees $(\deg R, \deg Q)$ are equal to either $(3,1)$, or $(4,1)$, or $(3,2)$. Then the polynomial $D_u(Q(x))$ has degree at most two, and \eqref{eq:generalformaux} takes the form
\begin{equation}\label{eq:generalformpellraw}
	ax^2+bx+c= - \Delta v^2,
\end{equation}
for some integers $a,b,c$.

\begin{proposition}\label{prop:infaux}
Assume that
\begin{itemize}
	\item[(a)] either $a=0$ or $-a\Delta$ is a positive integer that is not a perfect square, 
	\item[(b)] $b^2 - 4 a c \neq 0$, and
	\item[(c)] equation \eqref{eq:generalformpellraw} has an integer solution $(x_0,v_0)$ satisfying \eqref{eq:generalformcong}. 
\end{itemize}
Then equation \eqref{eq:generalformpellraw} has infinitely many integer solutions $(x,v)$ satisfying \eqref{eq:generalformcong}. 
\end{proposition}
\begin{proof}
	We will look for solutions in the form $v=v_0+2mw$ for some $w\in {\mathbb Z}$. Because $v_0$ satisfies \eqref{eq:generalformcong}, so does $v_0+2mw$ for every $w$. Substituting $v=v_0+2mw$ into \eqref{eq:generalformpellraw} results in a quadratic equation 
	\begin{equation}\label{eq:xwquad}
		ax^2+bx+c= - \Delta (v_0+2mw)^2
	\end{equation}
	in variables $(x,w)$, which has an integer solution $(x,w)=(x_0,0)$. We need to prove that \eqref{eq:xwquad} has infinitely many integer solutions. If $a\neq0$, then, viewed as a quadratic equation in the variables $(x,w)$, \eqref{eq:xwquad} has discriminant $16m^2a(-\Delta)$, which is positive and not a perfect square by (a), and it is nonsingular by (b). Hence the existence of one integer solution implies infinitely many integer solutions by Gauss's theorem \cite[Proposition 3.14]{MR4823864}. If $a=0$, then (b) implies that $b\neq 0$, and the question reduces to finding infinitely many integers $w$ such that $x=(-\Delta(v_0+2mw)^2-c)/b$ is an integer. Because $(x,w)=(x_0,0)$ is a solution to \eqref{eq:xwquad}, this is true for $w=0$, and therefore is true for every $w$ satisfying $w\equiv 0 \pmod{|b|}$.
\end{proof}

%\paragraph{Formalization of Proposition~\ref{prop:infaux}.}\label{par:lean42}
%The proof above retains the published conic theorem
%\cite[Proposition~3.14]{MR4823864}. The Lean proof instead completes the square
%and applies \texttt{genPell\_infinite\_cong}, a residue-controlled Pell result:
%a positive power of a Pell unit is congruent to $(1,0)$ modulo a prescribed
%modulus, and its iterates therefore preserve the seed's residue class.
%This auxiliary result is proved inside the formalization, not assumed as an axiom.
%It is not needed as a separate lemma in the manuscript.
%The assumption $\Delta\ne0$ is unnecessary for the infinitude of pairs in
%Proposition~\ref{prop:infaux} itself: when $\Delta=0$, condition (a) forces $a=0$,
%and the linear-case argument still gives infinitely many $v$ in the required
%residue class. In that case $x=-c/b$ is fixed; the assumption $\Delta\ne0$ remains
%needed in Proposition~\ref{prop:generalformmethod} to infer infinitely many distinct $x$.

All conditions in Proposition \ref{prop:infaux} are easy to check, including condition (c). 
For fixed $u,p,q$, put $M=2|m|>0$. Enumerate the finitely many residues
$j\in\{0,\ldots,M-1\}$ satisfying the two linear congruences
\eqref{eq:generalformcong}. For each such residue substitute $v=j+Mw$ into
\eqref{eq:generalformpellraw}, obtaining the binary quadratic equation
\[
 ax^2+bx+c=-\Delta(j+Mw)^2,\qquad (x,w)\in\mathbb Z^2.
\]
As mentioned in Section \ref{sec:method}, its solubility can be decided either by applying the general algorithm of Grunewald and Segal~\cite{MR2055712}, or by using well-known faster methods, see, e.g., \cite[Section 3.1]{MR4823864} and \cite{Lagarias1979}. For $a=0$, the verification is particularly simple: condition (b) gives $b\ne0$, and it suffices instead to enumerate $v$ modulo $\operatorname{lcm}(M,|b|)$ and test both \eqref{eq:generalformcong} and $b\mid-\Delta v^2-c$; then $x=(-\Delta v^2-c)/b$. These are complete tests for each fixed choice of $u,p,q$.

The described method is summarized in the following algorithm.
\begin{algorithm}\label{alg:generalform}
	~
	\begin{itemize}
		\item \textbf{Input:} Non-degenerate integral binary quadratic form $F$, and polynomials $R(t)$ and $Q(x)$ with integer coefficients such that $(\deg R, \deg Q)$ is equal to either $(3,1)$, or $(4,1)$, or $(3,2)$. 
		\item \textbf{Output:} If the algorithm terminates, then equation $F(y,z)=R(Q(x))$ is solvable in integers $(y,z)$ for infinitely many integers $x$. 
	\end{itemize}
	\begin{enumerate}
		\item Find integers $u,p,q$ such that $R(u)=F(p,q)\ne0$.
		\item Form the auxiliary equation \eqref{eq:generalformaux} and the congruences \eqref{eq:generalformcong}. Write  \eqref{eq:generalformaux} in the form \eqref{eq:generalformpellraw}. 
		\item Check whether conditions (a)--(c) in Proposition \ref{prop:infaux} are satisfied for  \eqref{eq:generalformpellraw}. If yes, STOP. Otherwise return to Step 1 and try a different triple $(u,p,q)$.
	\end{enumerate}
\end{algorithm}

Although Algorithm \ref{alg:generalform} is not guaranteed to terminate for all inputs due to the possible non-existence of suitable seed data, its power is that it does not impose any condition on the binary quadratic form $F$ except non-degeneracy. In particular, it does not require multiplicativity. Recall that a binary quadratic form $F$ is called multiplicative if, whenever it represents positive integers $m$ and $n$, it must represent their product $mn$. \cite[Section 5.4]{MR4823864} develops powerful methods for solving Problem \ref{prob:finite} for equations of the form $F(y,z)=P(x)$ for cubic $P$, but these methods are applicable only to multiplicative forms $F$. The simplest example of a form that is not multiplicative is
$$
F(y,z) = 2y^2 + yz + 2z^2.
$$
For example, it represents $2$ because $F(0,1)=2$, but does not represent $2 \cdot 2 = 4$, because equation $2y^2+yz+2z^2=4$ has no integer solutions: modulo $3$, its left-hand side is $-(y+z)^2$, which cannot be $1$. For this reason, as remarked in \cite[Section 5.4]{MR4823864}, their methods are not applicable to, for example, equations 
\begin{equation}\label{eq:2y2pyzp2z2}
	(a) \quad 2y^2+yz+2z^2 = x^3+1 \quad \text{and} \quad (b) \quad 2y^2+yz+2z^2 = x^3-1.
\end{equation}

Algorithm \ref{alg:generalform} solves Problem \ref{prob:finite} for these equations easily.

\begin{proposition}\label{prop:examples}
	 Both equations \eqref{eq:2y2pyzp2z2} have infinitely many integer solutions. 
\end{proposition}
\begin{proof}
	In this case, $A=C=2$, $B=1$, and $\Delta=-15$. For equation (a), take
$$
R(t)=t^3+1, \qquad Q(x)=x, \qquad u=1, \qquad (p,q)=(1,0).
$$
Then $m=R(u)=2=F(1,0)$, $r=R'(u)=3$, and $D_u(t)=t+2$.  The auxiliary equation \eqref{eq:generalformaux} is
$$
8x+7=15v^2.
$$
The congruences \eqref{eq:generalformcong} become
$$
3+v\equiv0\pmod 4,
\qquad
-4v\equiv0\pmod 4.
$$
They are therefore satisfied by every $v\equiv1\pmod4$.  For every such $v$, we have $v^2\equiv1\pmod8$, so the value
$$
x=\frac{15v^2-7}{8}
$$
is an integer.  Equivalently, putting $v=1+4n$ gives $x=30n^2+15n+1$ for any $n \in {\mathbb Z}$.
The recovery formulas~\eqref{eq:generalformYZ} give the explicit family
\[
\begin{aligned}
 x&=30n^2+15n+1,\\
 y&=30n^3+45n^2+15n+1,\\
 z&=-120n^3-90n^2-15n.
\end{aligned}
\]

Similarly, for equation (b), take
$$
R(t)=t^3-1, \qquad Q(x)=x, \qquad u=7, \qquad (p,q)=(3,12).
$$
Then $m=R(u)=342=F(3,12)$, $r=R'(u)=147$, and $D_u(t)=t+14$.  The auxiliary equation \eqref{eq:generalformaux} is
$$
456x-819=5v^2.
$$
The congruences \eqref{eq:generalformcong} are
$$
441+51v\equiv0\pmod {684},
\qquad
1764-24v\equiv0\pmod {684}.
$$
Both congruences are satisfied whenever $v\equiv45\pmod {228}$; for every $v$ in this residue class, the value
$$
x=\frac{5v^2+819}{456}
$$
is an integer.  Equivalently, putting $v=45+228n$ gives $x=570n^2+225n+24$ for any $n \in {\mathbb Z}$.
Here the recovery formulas give
\[
\begin{aligned}
 x&=570n^2+225n+24,\\
 y&=9690n^3+6105n^2+1189n+71,\\
 z&=-4560n^3-1230n^2+89n+29.
\end{aligned}
\]
In both cases $x$ is a non-constant polynomial in $n$, and therefore takes
infinitely many integer values. This proves the required infinitude.
\end{proof}

Algorithm \ref{alg:generalform} requires a quadratic form $F$ to be non-degenerate. In the degenerate case $\Delta=0$, condition $B^2-4AC=0$ implies that, for some integers $k,r,s$,
$$
(A,B,C)=(kr^2,2krs,ks^2).
$$
If $(A,B,C)=(0,0,0)$, then equation $F(y,z)=P(x)$ reduces to $P(x)=0$, and Problem \ref{prob:finite} for this equation is trivial. Otherwise $k\neq 0$ and $(r,s)\neq (0,0)$, and we have
$$
F(y,z) = Ay^2 + Byz + Cz^2 = k(ry+sz)^2,
$$  
hence any equation of the form $F(y,z)=P(x)$ reduces to
\begin{equation}\label{eq:degen}
 k t^2 = P(x),
\end{equation}
where $t=ry+sz$. There are general algorithms for describing all integer solutions $(x,t)$ to \eqref{eq:degen}, see \cite[Proposition 3.67]{MR4823864}. The original equation $F(y,z)=P(x)$ has infinitely many integer solutions if and only if \eqref{eq:degen} has a solution $(x_0,t_0)$ such that $t_0$ is divisible by $\gcd(r,s)$. 

It is interesting to compare the theory developed here with Section \ref{sec:method}. For $F(y,z)=y^2+z^2$, we have $\Delta=-4$, and \eqref{eq:generalformaux} becomes $4mD_u(Q(x))-r^2=4v^2$. Writing $V=2v$ gives the auxiliary equation \eqref{eq:auxquad2} of Section~\ref{sec:method}, up to the names of the variables. Conversely, an integer solution $(x,V)$ of \eqref{eq:auxquad2} has $V$ even, by reduction modulo~$4$, which allows us to reduce it back to \eqref{eq:generalformaux} with $\Delta=-4$.

The general construction in this section imposes the additional congruences \eqref{eq:generalformcong} to make its particular tangent-line coordinates integral. Section~\ref{sec:method} avoids them: identity~\eqref{eq:generalTangentIdentity} represents $4mR(Q(x))$, and property~(*) cancels the represented positive factor $4m$. It therefore proves the existence of an integral representation without requiring that particular tangent-line representation to be integral. We remark that this simplification uses both multiplicativity and the cancellation direction of (*); multiplicativity alone would not be sufficient for cancellation. For example, $4(y^2+z^2)$ is multiplicative and represents both $4$ and $4\cdot1$, but not~$1$. For binary quadratic forms satisfying the analogue of (*), such as $y^2+2z^2$, a theory similar to the one in Section~\ref{sec:method} can in principle be developed, but we decided instead to present the general Algorithm \ref{alg:generalform} applicable to all non-degenerate binary quadratic forms. 
%For another form whose represented positive integers satisfy (*), a suitable represented-product identity permits the same cancellation. 
%Even when only multiplicativity is available, a factorization $P(x)=P_1(x)P_2(x)$ can simplify an application: on any infinite set where both factors are positive and represented by the form, their product is represented, without a new tangent construction or new congruence conditions. The factorization used for the last equation in Corollary~\ref{prop:eq9194} is an instance of this simplification.

\end{document}